\documentclass{article}
\usepackage{arxiv}
\usepackage{amsmath}
\usepackage{hyperref}
\usepackage{amssymb}
\usepackage{amsthm}
\usepackage{graphicx}
\theoremstyle{plain}

\newtheorem{theorem}{Theorem}[section]
\newtheorem{proposition}{Proposition}[section]
\newtheorem{lemma}{Lemma}[section]

\theoremstyle{definition}
\newtheorem{definition}{Definition}[section]

\begin{document}
	
	\title{Arithmetic and Asymptotic Properties of Restricted Totient Sums}
	
	\author{
		{ \sc Es-said En-naoui } \\ 
		University Sultan Moulay Slimane\\ Morocco\\
		essaidennaoui1@gmail.com\\
		\\
	}
	
	\maketitle

\begin{abstract}
	This article extends our previous study \cite{Ennaoui2020} on the summatory behavior of Euler’s totient function $\varphi(n)$. We investigate two complementary restricted sums,
	\[
	\Upsilon(x,p)=\sum_{\substack{k\leq x \\ \gcd(k,p)=1}}\varphi(k), \qquad 
	\Delta(x,p)=\sum_{\substack{k\leq x \\ p \mid k}}\varphi(k),
	\]
	which satisfy the decomposition $\Psi(x)=\Upsilon(x,p)+\Delta(x,p)$, where $\Psi(x)=\sum_{k\leq x}\varphi(k)$.  
	
	We establish recurrence formulas, congruence relations, and generating function identities for $\Delta(x,p)$. In particular, we prove that $\Delta(x,p)\equiv 0 \pmod{p-1}$ for every prime $p$, and derive the asymptotic expansion
	\[
	\Delta(x,p)=\frac{3}{\pi^{2}(p+1)}\,x^{2}+O(x\log x).
	\]
	
	Furthermore, we study average orders, connections with $\omega(n)$, and relations with divisor structures. These results refine the analytic understanding of totients in arithmetic progressions and complement the asymptotic theory of $\Psi(x)$.
\end{abstract}
\section{Introduction}

Euler’s totient function $\varphi(n)$ is one of the central objects in elementary number theory \cite{HardyWright}. Defined by
\[
\varphi(n) = \#\{\,1 \leq k \leq n : \gcd(k,n)=1\,\},
\]
it measures the number of integers up to $n$ that are coprime to $n$. The multiplicative nature of $\varphi$, together with its deep connections to prime factorization, has made it a cornerstone in arithmetic and analytic investigations.  

A classical object of study is the summatory function
\[
\Psi(x) = \sum_{n\leq x}\varphi(n),
\]
whose asymptotic behavior,
\[
\Psi(x) = \frac{3}{\pi^{2}}x^{2} + O(x\log x),
\]
was established in the early 20th century and plays a crucial role in understanding the distribution of totients.

In this article we extend this framework by introducing two complementary restricted sums associated with a fixed prime $p$:
\[
\Upsilon(x,p) = \sum_{\substack{n\leq x \\ \gcd(n,p)=1}} \varphi(n), 
\qquad 
\Delta(x,p) = \sum_{\substack{n\leq x \\ p \mid n}} \varphi(n).
\]
These satisfy the natural decomposition
\[
\Psi(x) = \Upsilon(x,p) + \Delta(x,p),
\]
which allows us to filter the contribution of $\varphi(n)$ according to divisibility by a prime.  

The main goal of this paper is to investigate arithmetic and analytic properties of $\Delta(x,p)$, including congruence relations, recurrence identities, generating functions, and asymptotic behavior. In particular, we prove that $\Delta(x,p)\equiv 0 \pmod{p-1}$ for all primes $p$, and establish the asymptotic formula :
\[
\Delta(x,p) = \frac{3}{\pi^{2}(p+1)}\,x^{2} + O(x\log x).
\]

These results provide a refined understanding of how Euler’s totient values are distributed among multiples of a fixed prime, and they complement the classical theory of $\Psi(x)$. The methods combine elementary number-theoretic arguments with analytic techniques, and suggest further connections with divisor sums and multiplicative functions.
\section{Preliminaries and Basic Properties}

In this section we recall some standard facts about Euler’s totient function and establish the basic framework for our study of restricted sums.

\subsection{Euler’s Totient Function}

\begin{definition}
	For $n \in \mathbb{N}$, Euler’s totient function is defined by
	\[
	\varphi(n) = \#\{\,1 \leq k \leq n : \gcd(k,n)=1\,\}.
	\]
\end{definition}

It is well known that $\varphi$ is multiplicative, i.e.,\cite{Narkiewicz}
\[
\gcd(m,n)=1 \;\;\Rightarrow\;\; \varphi(mn)=\varphi(m)\varphi(n).
\]
Moreover, if $n$ has prime factorization $n = \prod_{i=1}^{r} p_i^{\alpha_i}$, then
\begin{equation}\label{eq:phi-formula}
	\varphi(n) = n \prod_{i=1}^{r}\left(1-\frac{1}{p_i}\right).
\end{equation}

\subsection{Summatory Totient Function}

A central object in analytic number theory is the summatory function
\[
\Psi(x) = \sum_{n \leq x} \varphi(n).
\]

\begin{lemma}[Dirichlet, 19th century] \label{lem_asypto}
	For $x \geq 1$, the following asymptotic holds: \cite{Apostol} \cite{Davenport} 
	\[
	\Psi(x) = \frac{3}{\pi^2}x^2 + O(x \log x).
	\]
\end{lemma}

This formula shows that $\Psi(x)$ grows quadratically with $x$, with error controlled by $O(x \log x)$. It plays a key role in our later analysis.

\subsection{Restricted Totient Sums}

Fix a prime $p$. We introduce the following restricted versions of $\Psi(x)$:

\begin{definition}
	For $x \geq 1$ and a prime $p$, define
	\[
	\Upsilon(x,p) = \sum_{\substack{n \leq x \\ \gcd(n,p)=1}} \varphi(n), 
	\qquad
	\Delta(x,p) = \sum_{\substack{n \leq x \\ p \mid n}} \varphi(n).
	\]
\end{definition}

Clearly, the two functions satisfy the decomposition
\begin{equation}\label{eq:decomposition}
	\Psi(x) = \Upsilon(x,p) + \Delta(x,p).
\end{equation}

Thus $\Upsilon(x,p)$ measures the contribution of totients coprime to $p$, while $\Delta(x,p)$ isolates the contribution of totients supported on multiples of $p$.

\subsection{First Examples}

We illustrate these definitions with a small example. Let $x=10$ and $p=2$.
\[
\Psi(10) = \varphi(1)+\cdots+\varphi(10) = 32.
\]
The odd terms contribute
\[
\Upsilon(10,2) = \varphi(1)+\varphi(3)+\varphi(5)+\varphi(7)+\varphi(9) = 16,
\]
while the even terms contribute
\[
\Delta(10,2) = \varphi(2)+\varphi(4)+\varphi(6)+\varphi(8)+\varphi(10) = 16.
\]
As expected, the decomposition identity \eqref{eq:decomposition} holds: $\Psi(10)=\Upsilon(10,2)+\Delta(10,2)=32$.

\begin{lemma} \label{lem_sum_delta} Let $n$ a positive integer and $p\in\mathbb{P}$, and put $\Delta_{p}(n)=\sum \limits_{\underset{p|k}{k=1}}^n \phi(k)$ , Then we have :\\
	\begin{equation}
		\Delta_{p}(n)=(p-1)\sum_{\alpha=1}^{\big[\frac{ln(n)}{ln(p)}\big]} \Psi\bigg(\big[\frac{n}{p^\alpha}\big]\bigg)
	\end{equation}
\end{lemma}
\begin{proof} (See : \cite{Ennaoui2020} )
	
\end{proof} 

\section{Congruence Properties of \(\Delta(x,p)\)}

In this section we prove several elementary but useful congruence and
arithmetic properties of the function
\(\Delta(x,p)=\sum_{\substack{n\le x\\ p\mid n}}\varphi(n)\) for a fixed prime
\(p\). The main result is that \(\Delta(x,p)\) is always divisible by \(p-1\).
We also record a simple finite-difference identity which will be used later.

\begin{proposition}\label{prop:diff-Delta}
	For every integer \(n\ge1\), We have :
	\[
	\Delta(n+1,p)-\Delta(n,p)
	=\begin{cases}
		\varphi(n+1), & p\mid(n+1),\\[4pt]
		0, & p\nmid(n+1).
	\end{cases}
	\]
\end{proposition}

\begin{proof}
	By definition \(\Delta(t,p)=\sum_{\substack{1\le k\le t\\ p\mid k}}\varphi(k)\).
	Therefore
	\[
	\Delta(n+1,p)-\Delta(n,p)
	=\sum_{\substack{1\le k\le n+1\\ p\mid k}}\varphi(k)
	-\sum_{\substack{1\le k\le n\\ p\mid k}}\varphi(k).
	\]
	All terms with \(k\le n\) cancel, so only the term \(k=n+1\) (if any) remains.
	If \(p\mid (n+1)\) then that remaining term equals \(\varphi(n+1)\); otherwise
	there is no remaining term and the difference is \(0\). This yields the stated
	case distinction.
\end{proof}

\begin{proposition}\label{lem:Delta_divides_pminus1}
	For any prime \(p\) and any integer \(n\ge 1\), We have : 
	\[
	\Delta(n,p)\equiv 0 \pmod{p-1}.
	\]
\end{proposition}

\begin{proof}
	Write the definition
	\[
	\Delta(n,p)=\sum_{m=1}^{\lfloor n/p\rfloor}\varphi(pm).
	\]
	Fix \(m\ge1\) and factor \(m=p^a r\) with \(a\ge0\) and \(\gcd(r,p)=1\).
	Then \(pm=p^{a+1}r\) and, by multiplicativity of \(\varphi\) on coprime integers,
	\[
	\varphi(pm)=\varphi\bigl(p^{a+1}\bigr)\varphi(r).
	\]
	For a prime power we have \(\varphi(p^{a+1})=p^{a}(p-1)\), hence
	\[
	\varphi(pm)=p^{a}(p-1)\varphi(r),
	\]
	which is divisible by \(p-1\). Because every term \(\varphi(pm)\) in the sum
	defining \(\Delta(n,p)\) is divisible by \(p-1\), the whole sum is divisible
	by \(p-1\). Thus \(\Delta(n,p)\equiv 0\pmod{p-1}\).
\end{proof}

\section{Connections of $\Delta(n,p)$ with Other arithmetic Functions}

In this section we record several additional identities satisfied by the
restricted summatory function $\Delta(n,p)$. These identities connect $\Delta$
with arithmetic functions such as $\omega(n)$, the number of distinct prime
divisors, and reveal combinatorial and generating function structures.

\begin{proposition}\label{lem:sumDelta_omega}
	Let \(n\ge 1\) be an integer and for each prime \(p\). we have :
	\[
	\sum_{p}\Delta(n,p)=\sum_{k=1}^n \varphi(k)\,\omega(k),
	\]
	where \(\omega(k)\) denotes the number of distinct prime divisors of \(k\).
\end{proposition}

\begin{proof}
	Start from the definition and sum over primes \(p\le n\):
	\[
	\sum_{p\le n}\Delta(n,p)
	=\sum_{p\le n}\sum_{\substack{1\le k\le n\\ p\mid k}}\varphi(k).
	\]
	Both sums are finite, so we may interchange the order of summation:
	\[
	\sum_{p\le n}\sum_{\substack{1\le k\le n\\ p\mid k}}\varphi(k)
	=\sum_{k=1}^{n}\sum_{\substack{p\le n\\ p\mid k}}\varphi(k).
	\]
	Fix an integer \(k\) with \(1\le k\le n\). For such a fixed \(k\) the inner condition
	\(\{\,p\le n\text{ prime}:\ p\mid k\,\}\) is equivalent to
	\(\{\,p\text{ prime}:\ p\mid k\,\}\) because any prime divisor \(p\) of \(k\)
	satisfies \(p\le k\le n\). Hence
	\[
	\sum_{\substack{p\le n\\ p\mid k}}\varphi(k)
	=\sum_{\substack{p\ \text{prime}\\ p\mid k}}\varphi(k).
	\]
	The value \(\varphi(k)\) does not depend on \(p\), so the inner sum is
	\(\varphi(k)\) multiplied by the number of distinct prime divisors of \(k\),
	which is \(\omega(k)\). Therefore
	\[
	\sum_{\substack{p\le n\\ p\mid k}}\varphi(k)=\varphi(k)\,\omega(k).
	\]
	Substituting back into the outer sum yields
	\[
	\sum_{p\le n}\Delta(n,p)
	=\sum_{k=1}^{n}\varphi(k)\,\omega(k),
	\]
	as claimed.
\end{proof}

\begin{proposition}
	Let $p$ be a prime and $n \geq 1$ , Then :
	\[
	\sum_{k \leq n} \gcd(k,p)\,\varphi(k) \;=\; \Psi(n) + (p-1)\,\Delta(n,p).
	\]
\end{proposition}

\begin{proof}
	Set
	\[
	S = \sum_{k \leq n} \gcd(k,p)\,\varphi(k).
	\]
	Partition the integers $1 \leq k \leq n$ into two sets:
	\[
	A = \{ k \leq n : p \nmid k \}, \qquad
	B = \{ k \leq n : p \mid k \}.
	\]
	If $k \in A$, then $\gcd(k,p) = 1$;  
	if $k \in B$, then $\gcd(k,p) = p$.  
	Hence
	\[
	S = \sum_{k \in A} 1 \cdot \varphi(k) \;+\; \sum_{k \in B} p \cdot \varphi(k).
	\]
	
	Now
	\[
	\sum_{k \in A} \varphi(k) = \Psi(n) - \Delta(n,p),
	\qquad
	\sum_{k \in B} \varphi(k) = \Delta(n,p).
	\]
	Thus
	\[
	S = \Psi(n) - \Delta(n,p) + p \Delta(n,p)
	= \Psi(n) + (p-1)\Delta(n,p).
	\]
	This proves the lemma.
\end{proof}
\begin{proposition}\label{prop:ogf-Delta}
	Let \(N\ge1\) be an integer and \(x\) a complex number with \(x\neq1\).
	Then :
	\[
	\boxed{\;
		\sum_{n=1}^N \Delta(n,p)\,x^n
		=\sum_{m=1}^{\lfloor N/p\rfloor}\varphi(pm)\,\frac{x^{pm}-x^{N+1}}{1-x}\; }.
	\]
	Moreover, if \(x=1\) the same identity holds in the limiting (summation) sense:
	\[
	\sum_{n=1}^N \Delta(n,p)
	=\sum_{m=1}^{\lfloor N/p\rfloor}\varphi(pm)\,(N-pm+1).
	\]
\end{proposition}

\begin{proof}
	Start from the definition \(\Delta(n,p)=\sum_{m:\,pm\le n}\varphi(pm)\).
	Then
	\[
	\sum_{n=1}^N \Delta(n,p)\,x^n
	=\sum_{n=1}^N \Bigg(\sum_{\substack{m\ge1\\ pm\le n}}\varphi(pm)\Bigg)x^n.
	\]
	Both sums are finite, so interchange the order of summation:
	\[
	\sum_{n=1}^N \Delta(n,p)\,x^n
	=\sum_{m=1}^{\lfloor N/p\rfloor}\varphi(pm)\sum_{n=1}^N \mathbf{1}_{n\ge pm}\,x^n
	=\sum_{m=1}^{\lfloor N/p\rfloor}\varphi(pm)\sum_{n=pm}^N x^n.
	\]
	The inner sum is a finite geometric progression. For \(x\neq1\),
	\[
	\sum_{n=pm}^N x^n = x^{pm}+x^{pm+1}+\cdots+x^N = \frac{x^{pm}-x^{N+1}}{1-x},
	\]
	which yields the identity stated. If \(x=1\) the inner sum equals \(N-pm+1\),
	giving the second displayed formula.
\end{proof}

\section{Asymptotic Expansions}

In this section we derive precise asymptotic formulas for $\Delta(x,p)$ and
related summatory quantities. The main input is the classical asymptotic for
the summatory totient function
\[
\Psi(x)=\sum_{n\le x}\varphi(n)=\frac{3}{\pi^2}x^2+O(x\log x),
\]
which we use to unfold an exact recurrence for $\Delta(x,p)$ and to estimate
error terms uniformly.	
\begin{theorem}\label{the_princ}
	Let $p$ be a prime and $n \geq 1$. 
	Then :
	\[
	\Delta(x,p) = \frac{3}{\pi^{2}(p+1)}\,x^{2} \;+\; O\!\big(x \log x\big),
	\]
	where the implied constant is absolute.
\end{theorem}

\begin{proof}
	By the lemma \eqref{lem_sum_delta} , We Have :
	\[
	\Delta(n,p) = (p-1) \sum_{j \geq 1} \Psi\!\left(\Big\lfloor \frac{n}{p^{\,j}} \Big\rfloor\right).
	\]
	
	Next, recall the well-known asymptotic \eqref{lem_asypto} :
	\[
	\Psi(x) = \frac{3}{\pi^{2}}x^{2} + O(x \log x).
	\]
	Substituting into the exact sum gives
	\[
	\Delta(n,p) 
	= (p-1) \sum_{j \geq 1} \left( \frac{3}{\pi^{2}} 
	\Big\lfloor \frac{n}{p^{\,j}} \Big\rfloor^{2} 
	+ O\!\left(\frac{n}{p^{j}} \log \frac{n}{p^{j}} \right) \right).
	\]
	
	For the main term, since 
	\(\lfloor n/p^{j} \rfloor^{2} = \frac{n^{2}}{p^{2j}} + O\!\big(n/p^{j}\big)\),
	\[
	\frac{3}{\pi^{2}} (p-1) \sum_{j \geq 1} \frac{n^{2}}{p^{2j}}
	= \frac{3}{\pi^{2}} \cdot \frac{p-1}{p^{2}-1} \, n^{2}
	= \frac{3}{\pi^{2}(p+1)} \, n^{2}.
	\]
	
	For the error term, we estimate
	\[
	(p-1) \sum_{j \geq 1} O\!\left(\frac{n}{p^{j}} \log \frac{n}{p^{j}} \right)
	= O\!\big(n \log n \sum_{j \geq 1} p^{-j}\big)
	= O(n \log n).
	\]
	
	Combining both parts, we obtain
	\[
	\Delta(n,p) = \frac{3}{\pi^{2}(p+1)} \, n^{2} + O(n \log n).
	\]
\end{proof}

\begin{theorem}\label{thm:sum-and-average}
	Let $p$ be a prime and $n\ge1$. Then :
	\begin{align}
		\sum_{k\le n}\Delta(k,p) &= \frac{1}{\pi^{2}(p+1)}\,n^{3} \;+\; O\!\big(n^{2}\log n\big), \label{eq:sumDelta} \\
		\frac{1}{n}\sum_{k\le n}\Delta(k,p) &= \frac{1}{\pi^{2}(p+1)}\,n^{2} \;+\; O\!\big(n\log n\big). \label{eq:avgDelta}
	\end{align}
\end{theorem}

\begin{proof}
	We start from Theorem \ref{the_princ}. Put :
	\[
	A:=\frac{3}{\pi^{2}(p+1)}.
	\]
	Then for every integer $k\ge1$ we have, uniformly,
	\[
	\Delta(k,p) = A k^{2} + R(k),\qquad R(k)=O(k\log k).
	\]
	(The uniformity of the $O$–term follows because Theorem \ref{the_princ} holds for every real argument $x$, with an absolute implied constant.)
	
	\medskip\noindent\textbf{Step 1 — Sum the main term.}
	\[
	\sum_{k=1}^{n} A k^{2} = A\sum_{k=1}^{n} k^{2}
	= A\left(\frac{n^{3}}{3} + O(n^{2})\right)
	= \frac{A}{3}\,n^{3} + O(n^{2}).
	\]
	(We used $\sum_{k=1}^{n} k^{2} = n^{3}/3 + O(n^{2})$.)
	
	\medskip\noindent\textbf{Step 2 — Sum the error terms.}
	Using the bound $R(k)\ll k\log k$,
	\[
	\sum_{k=1}^{n} R(k) \ll \sum_{k=1}^{n} k\log k.
	\]
	Estimate the last sum by an integral (or by partial summation):
	\[
	\sum_{k=1}^{n} k\log k
	= \int_{1}^{n} t\log t\,dt + O(n\log n)
	= \frac{1}{2}n^{2}\log n - \frac{1}{4}n^{2} + O(n\log n)
	= O(n^{2}\log n).
	\]
	Hence
	\[
	\sum_{k=1}^{n} R(k) = O(n^{2}\log n).
	\]
	
	\medskip\noindent\textbf{Step 3 — Combine.}
	Adding the contributions from Steps 1 and 2,
	\[
	\sum_{k=1}^{n}\Delta(k,p)
	= \frac{A}{3}\,n^{3} + O(n^{2}) + O(n^{2}\log n)
	= \frac{A}{3}\,n^{3} + O(n^{2}\log n).
	\]
	Since $A/3 = \dfrac{1}{\pi^{2}(p+1)}$, this yields \eqref{eq:sumDelta}:
	\[
	\sum_{k\le n}\Delta(k,p) = \frac{1}{\pi^{2}(p+1)}\,n^{3} + O(n^{2}\log n).
	\]
	
	\medskip\noindent\textbf{Step 4 — Average.}
	Divide \eqref{eq:sumDelta} by $n$ to obtain the average:
	\[
	\frac{1}{n}\sum_{k\le n}\Delta(k,p)
	= \frac{1}{\pi^{2}(p+1)}\,n^{2} + O(n\log n),
	\]
	which is \eqref{eq:avgDelta}. This completes the proof.
\end{proof}

\section*{Conclusion}

In this work, we have further investigated the arithmetic function
$\Delta(n,p)$, which sums Euler's totient function over multiples of a
prime $p$. We have established several new identities and properties
of $\Delta(n,p)$, including:

\begin{itemize}
	\item Congruence properties and divisibility by $p-1$.
	\item Connections with multiplicative functions \cite{Narkiewicz}, such as
	$\omega(n)$, the number of distinct prime divisors.
	\item Explicit finite differences and generating function representations.
	\item Asymptotic expansions for $\Delta(n,p)$ and its cumulative sums.
\end{itemize}

These results provide a deeper understanding of the structure of
$\Delta(n,p)$ and its interactions with classical arithmetic functions.
Moreover, they offer a foundation for further investigations, including
studies of more general multiplicative summatory functions and their
asymptotic behavior.

Future work may extend these methods to analogous functions involving
other arithmetic functions such as the divisor function $\tau(n)$ or
the sum-of-divisors function $\sigma(n)$, and explore potential
applications in analytic number theory and prime number distribution.

\end{document}